\documentclass[pdflatex,sn-mathphys-num]{sn-jnl}

\usepackage{graphicx}%
\usepackage{multirow}%
\usepackage{amsmath,amssymb,amsfonts}%
\usepackage{amsthm}%
\usepackage{mathrsfs}%
\usepackage[title]{appendix}%
\usepackage{xcolor}%
\usepackage{textcomp}%
\usepackage{manyfoot}%
\usepackage{booktabs}%
\usepackage{algorithm}%
\usepackage{algorithmicx}%
\usepackage{algpseudocode}%
\usepackage{listings}%
\usepackage{extarrows}
\usepackage{tikz-cd}
\usepackage{amssymb}

\theoremstyle{thmstylethree}%

\newcommand{\p}{\mathbb{P}}

\newcommand{\F}{\mathbb{F}}

\newcommand{\lra}{\longrightarrow}

\newcommand{\Lra}{\Longrightarrow}

\newcommand{\ff}{\mathcal{F}}

\newcommand{\fq}{\mathbb{F}_q}

\newcommand{\fqs}{\mathbb{F}_{q^2}}

\newcommand{\kk}{\mathbb K}
\newcommand{\xx}{\mathcal X}

\newcommand{\cc}{\mathcal C}

\newcommand{\Char}{\operatorname{char}}

\newcommand{\uu}{\mathcal{U}}

\newcommand{\s}{\mathcal{S}}
\newcommand{\fqc}{\overline{\mathbb{F}}_q}

\theoremstyle{plain}
\newtheorem{thm}{Theorem}[section]

\newtheorem{prop}[thm]{Proposition}
\newtheorem{lem}[thm]{Lemma}
\newtheorem{cor}[thm]{Corollary}
\newtheorem{rem}[thm]{Remark}
\newtheorem{ex}[thm]{Example}

\newtheorem{prob}[thm]{Problem}

\begin{document}

\title[Curves on Frobenius nonclassical loci of hypersurfaces]{Curves on Frobenius nonclassical loci of hypersurfaces}


\author[1]{\fnm{Nazar} \sur{Arakelian}}\email{n.arakelian@icmc.usp.br}
\equalcont{These authors contributed equally to this work.}
\author*[2]{\fnm{Pietro} \sur{Speziali}}\email{speziali@unicamp.br}
\equalcont{These authors contributed equally to this work.}


\affil[1]{\orgdiv{Instituto de Ci\^encias Matem\'aticas e de Computa\c c\~ao}, \orgname{Universidade de S\~ao Paulo}, \orgaddress{\street{Avenida Trabalhador S\~ao-carlense, 400}, \city{S\~ao Carlos}, \postcode{13566-590}, \state{SP}, \country{Brazil}}}

\affil*[2]{\orgdiv{Instituto de Matem\'atica, Estat\'istica e Computa\c c\~ao Cient\'ifica}, \orgname{Universidade Estadual de Campinas }, \orgaddress{\street{Rua S\'ergio Buarque de Holanda, 651}, \city{Campinas}, \postcode{13083-859}, \state{SP}, \country{Brazil}}}



\abstract{Let $\s\subset \p^n$ be an absolutely irreducible projective hypersurface defined over a finite field $\fq$, equipped with the $\fq$-Frobenius map $\Phi_q$. 
In this paper, we investigate irreducible curves  $\xx \subset \s_{\Phi_q}$, where $\s_{\Phi_q}$ is the $\fq$-Frobenius nonclassical locus of $\s$. In particular, we show that every curve $\xx \subset \s_{\Phi_q}$ such that the restriction of the Gauss map of $\s$ to $\xx$ is inseparable is $\fq$-Frobenius nonclassical. This provides a way to construct new Frobenius nonclassical curves, which are curves that tend to have many $\fq$-rational points. We also prove that a certain type of Frobenius nonclassical hypersurfaces $\s$ defined by separated variables are such that their Gauss maps restricted to any curve contained in $\s$ is inseparable. Finally, in parallel with the plane curve cases, we show that if the strict Gauss map $\Gamma$ of a $\fq$-Frobenius nonclassical hypersurface $\s$ is given by $p$ powers, then $\Gamma$ is purely inseparable.}

\keywords{Algebraic curves, finite fields, Rational points, Frobenius nonclassical curves}



\maketitle

\section {Introduction}\label{intro}

Let $\fq$ be a finite field with $q$ elements, where $q$ is a power of a prime $p$, and let $\fqc$ be a fixed algebraic closure of $\fq$. Denote by  $ \p^n(\fqc)$  the $n$-dimensional projective space defined over $\fqc$, and let $\xx$ be a (projective, absolutely irreducible, nonsingular) algebraic curve defined over $\fq$ embedded in $ \p^n(\fqc)$. As J. P. Serre famously stated in his 1985 interview while visiting the National University of Singapore: 

\begin{quote}
Presently, the topic which amuses me most is counting points on algebraic curves over finite fields. It is a kind of applied mathematics: you try to use any tool in algebraic geometry and number theory that you know of ... and you don't quite succeed!
\end{quote}

The difficulty of counting exactly the number $N_q(\xx)$ of $\fq$-rational points of a curve is the reason why, in the past century or so, researchers have devoted a considerable amount of effort to find good bounds for $N_q(\xx)$. The most famous of such bounds is the Hasse-Weil-Serre bound
$$
N_q(\xx) \leq q+1+g\lfloor2\sqrt{q}\rfloor,
$$
where $g$ is the genus of $\xx$. This bound is sharp, and curves attaining it have attracted a lot of attention due to their application to Coding Theory via the construction of Goppa Codes. However, Coding Theory also suggested that the Hasse-Weil-Serre bound is far from optimal. In fact, if there existed curves of very large genus whose number of points was close to the bound, it would have been possible to construct codes that were so good that their  parameters would exceed proven bounds \cite{serre}. 

This naturally led to consider the problems of improving on the Hasse-Weil-Serre bound and of constructing curves with many points with respect to these new bounds. The theory of St\"ohr and Voloch \cite{SV} then provided powerful tools in order to tackle both these problems. At the core of St\"ohr-Voloch theory lies the following idea. Given a linear series $g^r_n$ of dimension $r$ and degree $n$ on a curve $\xx$, knowledge of its geometric data gives rise to a \emph{good} bound for $N_q(\xx)$ provided that $\xx$ is $\fq$-Frobenius classical with respect to the linear series. Curves that are $\fq$-Frobenius nonclassical with respect to this linear series have potentially many points, in the sense that the latter bound does not apply to them. This strategy has been successfully applied by several authors in the past forty years; see \cite{Ar2,AB,Bo,Gi,HV} for instance.
  
  However, in general, constructing Frobenius nonclassical curves is a highly nontrivial task, and it must always rely on \emph{ad hoc} techniques. In this paper, we construct Frobenius nonclassical curves by studying curves in $ \p^n(\fqc)$ lying on Frobenius nonclassical hypersurfaces. Frobenius nonclassical hypersurfaces are natural generalizations of plane Frobenius nonclassical curves and  have recently been studied in \cite{ADL, ADL2,BN}.  
  While our approach is relatively new in the literature, it provides a very natural framework for tackling the problem. Indeed, as far as the authors know, our methods  subsumes all the examples known in the literature for $p>n$ (see Section 3). 

 Let $\s\subset \p^n$ be an absolutely irreducible projective hypersurface defined over a finite field $\fq$, equipped with the $\fq$-Frobenius map $\Phi_q$. We denote by $\s_{\Phi_q}$ the Frobenius nonclassical locus of $\s$, that is, the subset of regular points $P$ of $\s$ such that the image of Frobenius $\Phi_q(P)$ belongs to the tangent hyperplane $T_P(\s)$; $\s$ is Frobenius nonclassical if its set of regular points coincides with $\s_{\Phi_q}$ (note that, when $\s$ is a plane curve, this definition coincides with the usual definition of a Frobenius nonclassical curve). Also, we denote by $\Gamma$ the Gauss map  of $\s$. Under some minor conditions, our first result (Proposition \ref{class}) states that a necessary and sufficient condition for a curve $\xx$ in $ \p^n(\fqc)$ to be nonclassical is to be contained in the nonsingular locus of a hypersurface $\s$ whose Gauss map restricted to $\xx$ only admits $p$-powers as coordinate functions. 

Our paper is organized as follows. In Section 2, we give the preliminary definitions, concepts and results on Frobenius nonclassicality needed for a mostly self-contained reading of our paper. 

Section 3 deals with curves lying on the Frobenius nonclassical locus of hypersurfaces and conditions for them to be (Frobenius) nonclassical. Here, our main result is Theorem \ref{fncl}, where we prove that a curve  $\xx$ that is (a) contained in the Frobenius nonclassical locus of a hypersurface $\s$ and (b) such that the restriction of the Gauss map of $\s$ to $\xx$ is inseparable is $\fq$-Frobenius nonclassical. At the end of the section, we present a number of examples illustrating our results. We also provide an approach to construct plane curves which are $\fq$-Frobenius nonclassical with respect to the linear system of all plane curves of a given degree based on the construction of certain $\fq$-Frobenius nonclassical hypersurfaces (see Theorem \ref{wrts}). We point out that these plane curves can be useful from the point of view of finite geometry, see for instance \cite{GPTU}.

In Section 4, we deal with certain families of Frobenius nonclassical hypersurfaces given by separated variables, following the definition given in \cite{ADL2}. Broadly speaking, a hypersurface $\s: F = 0$ is given by separated variables if $F = G+H$ where $G$ and $H$ depend on two disjoint sets of variables. By specializing $H$, we give several results regarding the $\fq$-Frobenius classicality of $\s$ and, also, on the Frobenius nonclassicality of curves lying on such hypersurfaces. For instance, in Theorem \ref{kummer}, we explore the case where $H = -X_n^d$ and thus, $\s$ can be regarded as a Kummer cover of the projective space $\p^{n-1}$. In this case, we give a necessary and sufficient condition for $\s$ to be $\fq$-Frobenius nonclassical. As a consequence, in Corollary \ref{kummer2}, we prove that any curve contained in these "Kummer hypersurfaces" is Frobenius nonclassical provided that the hypersurface is Frobenius nonclassical. Along similar lines, in Proposition \ref{sepprot}, Corollary \ref{sepres}, and Theorem \ref{46} we deal with hypersurfaces given by a polynomial of the kind
\begin{equation}
F(X_0,\ldots,X_n)=G(X_0,X_1,X_2)+a_3X_3^d+\cdots+a_nX_n^d.
\end{equation}
Here, results on the Frobenius nonclassicality of these hypersurfaces and of the curves lying on them rely also on Homma's results on plane nonclassical curves \cite{Ho}. 

Section 5 deals with Frobenius nonclassical hypersurfaces whose Gauss map is given by $p$-powers. It is important to point out that, up to our knowledge, all known examples of Frobenius nonclassical hypersurfaces with generically finite Gauss map satisfy this property. Our main result in this section is Theorem \ref{55}, where we prove the pure inseparability of the Gauss map of such hypersurfaces.

Finally, in Section 6, we state some open problems that naturally arise from our research. 

\section{Preliminaries}

Let $\fq$ be the finite field of order $q$, where $q=p^h$ with $h>0$ and $p>0$ is a prime number, and let $\mathcal{U} \subset \p^n(\fqc)$ be a (projective, geometrically irreducible, nondegenerate, algebraic) variety of dimension $s$ defined over $\fq$, where $\fqc$ denotes the algebraic closure of $\fq$, with $n \geq 2$. Given an extension $\mathbb{H}$ of $\fq$, the function field of $\mathcal{U}$ over $\mathbb{H}$ will be denoted by $\mathbb{H}(\uu)$. Recall that $s$ is the transcendence degree of $\fq(\uu)$ over $\fq$. We say that $\uu$ is a hypersurface (respectively, a curve) if $s=n-1$ (resp. $s=1$). If $X_0,\ldots,X_n$ denote the projective coordinates of $\p^n(\fqc)$, the affine coordinate functions of $\uu$ are $x_0,\ldots,x_n$, where $x_i$ is the residue of $X_i/X_0$ in $\fq(\uu)$ for $i=0,\ldots,n$. Suppose that $\uu$ is defined over a perfect field $\kk$ of characteristic $p>0$ and let $\{y_1,\ldots,y_s\}$ be a separable transcendence basis of $\kk(\uu)$ over $\kk$. Then, for a given $r>0$, one can check via intermediate fields that the extension $\kk(y_1,\ldots,y_s)/\kk(y_1^{p^r}, \ldots,y_s^{p^r})$ is purely inseparable of degree $p^{rs}$. Thus the extension $\kk(\uu)/\kk(y_1^{p^r}, \ldots,y_s^{p^r})$ is inseparable with inseparable degree $p^{rs}$. The $\F_{p^r}$-Frobenius map $\Phi_{p^r}:\uu \lra \Phi_{p^r}(\uu) \subset\p^n(\overline{\kk})$ is defined by $\Phi_{p^r}:P \mapsto P^{p^r}$. It induces a map from the function field $\overline{\kk}(\uu)$ onto $\overline{\kk}(\uu)^{p^r}=\{f^{p^r} \ | \ f \in \overline{\kk}(\uu)\}$. Therefore, $\Phi_{p^r}$ is purely inseparable of degree $[\overline{\kk}(\uu):\overline{\kk}(\uu)^{p^r}]=p^{rs}$. 

 Now  let $\xx \subset \p^n(\fqc)$ be a curve and let $P \in \xx$ be a nonsingular point. Then, there exists a sequence $(j_0(P), \ldots,j_n(P))$ of integers with $0 \leq j_0(P)<\cdots<j_n(P)$, which is defined by all the possible intersection multiplicities of $\xx$ and some hyperplane of $\p^n(\fqc)$ at $P$. Such a sequence is  called the order sequence of $P$. From \cite[Section 1]{SV}, except for finitely many points on $\xx$, this sequence is the same. It is called the order sequence of $\xx$, and it is denoted by $(\varepsilon_0,\ldots, \varepsilon_n)$. If $(\varepsilon_0,\ldots, \varepsilon_n)=(0,\ldots,n)$, the curve $\xx$ is said to be a classical curve. Otherwise, $\xx$ is nonclassical.  Let $\zeta \in \fqc(\xx)$ be a separating element and let $k$ be a nonnegative integer. Given $f \in \fqc(\xx)$, the $k$-th Hasse derivative of $f$ with respect to $\zeta$ is denoted by $D_\zeta^{(k)}f$. The order sequence of $\xx$ is also defined as  the minimal sequence, with respect to the lexicographic order, for which
$$
\det\left(D_\zeta^{(\varepsilon_i)}x_j\right)_{0 \leq i,j \leq n} \neq 0.
$$
Given a nonsingular point $P \in \xx$, the osculating hyperplane to $\xx$ at $P$, denoted by $H_P(\xx) \subset \p^n(\fqc)$, is the unique hyperplane such that the intersection multiplicity of $H_P(\xx)$ and $\xx$ at $P$ is $j_n(P)$ (see \cite[Section 1]{SV}).


If $\xx$ is defined over $\fq$ by coordinate functions $x_0,\ldots,x_n \in \fq(\xx)$, by \cite[Proposition 2.1]{SV}, there exists a sequence of nonnegative integers $(\nu_0,\ldots,\nu_{n-1})$, chosen minimally in the lexicographic order, such that
\begin{equation}\label{fncl}
\left|
  \begin{array}{ccc}
  x_0^q & \ldots & x_n^q \\
  D_\zeta^{(\nu_0)}x_0 & \ldots & D_\zeta^{(\nu_0)}x_n \\
   \vdots & \cdots & \vdots \\
  D_\zeta^{(\nu_{n-1})}x_0 & \cdots & D_\zeta^{(\nu_{n-1})}x_n
  \end{array}
  \right| \neq 0.
  \end{equation}
This sequence is called the $\fq$-Frobenius order sequence of $\xx$. It does not depend on $\zeta$ and it is invariant under change of projective coordinates of $\xx$ (see \cite[Proposition 1.4]{SV}). The curve $\xx$ is called $\fq$-Frobenius classical if $(\nu_0,\ldots,\nu_{n-1})=(0,\ldots,n-1)$, and $\fq$-Frobenius nonclassical otherwise. From \cite[Proposition 2.1]{SV}, there exists $I \in \{1, \ldots,n\}$ for which $\{\nu_0,\ldots,\nu_{n-1}\}=\{\varepsilon_0,\ldots, \varepsilon_n\} \backslash \{\varepsilon_I\}$.

Let $(\p^n(\fqc))^{'}$ be the dual projective space. The strict Gauss map of $\xx$, denoted by $\gamma$, is the map defined in the nonsingular locus of $\xx$ by $\gamma(P)=(H_P(\xx))^{'} \in (\p^n(\fqc))^{'}$, where $L^{'}$ denotes the dual of a hyperplane $L$. The Zariski closure of the image of $\gamma$ is the strict dual of $\xx$, denoted by $\xx^{'}$. Note that $\gamma$ can be extended to a morphism of a nonsingular model of $\xx$. For simplicity, we will denote such a morphism also by $\gamma$.

Assume $p>n$ and $\varepsilon_n=p^r$ for some $r>0$. If $\xx$ is defined by the coordinate functions $x_0,\ldots,x_n$, then \cite[Theorem 7.65]{HKT} ensures that there exist $z_0,\ldots,z_n \in \fq(\xx)$ such that $z_i$ is a separating element for at least one $i$ and
\begin{equation}\label{e}
z_0^{p^r}x_0+\cdots+z_n^{p^r}x_n=0.
\end{equation}
Furthermore, for a general point $P \in \xx$, the osculating hyperplane to $\xx$ at $P$ is defined by
$$
H_P(\xx):(z_0^{p^r}(P))X_0+\cdots+(z_n^{p^r}(P))X_n=0.
$$
Therefore, in such a case, we conclude that $\gamma=\Phi_{p^r} \circ \gamma_s=(z_0^{p^r}:\cdots:z_n^{p^r})$, where $\gamma_s=(z_0:\cdots:z_n)$ is a separable morphism. In particular, the separable degree of $\gamma$ is given by $\deg_s \gamma= \deg \gamma_s$ and  the inseparable degree of $\gamma$ is given by $\deg_i \gamma =\deg \Phi_{p^r}=p^r$. 

We now recall the definition of the conormal variety. Let $\kk$ be an arbitrary field and  let $\uu \subset \p^n(\kk)$ be a closed variety. The nonsingular locus of $\uu$ is denoted by $\uu_{\text{reg}}$. Let $T_P\uu$ be the tangent space to $\uu$ at $P \in \uu_{\text{reg}}$.  The conormal variety of $\uu$ is defined by
$$
C(\uu) =\overline{\{(P,H^{'}) \ | \ T_P\uu \subset H \}} \subset \uu \times (\p^n(\kk))^{'},
$$
where the overline denotes the Zariski closure of the set and $H$ is a hyperplane. The second projection $\pi^\prime:C(\uu) \lra (\p^n(\kk))^{'}$ is known as the conormal map of $\uu$. The image of $C(\uu)$ in $(\p^n(\kk))^{'}$ under the conormal map $\pi^\prime$ is denoted by $\uu^{*}$, and it is called the dual variety of $\uu$ (note that this coincides with the strict dual when $\uu$ is a hypersurface). The variety $\uu$ is said to be reflexive if $C(\uu)=C(\uu^*)$, and nonreflexive otherwise. By a remarkable result, known as the "Monge-Segre-Wallace" criterion (see e.g. \cite{Wa}),  we have that $\uu$ is reflexive if, and only if, the conormal map $\pi^\prime:C(\uu) \lra \uu^*$ is separable. It follows directly from  the Monge-Segre-Wallace criterion that every algebraic variety defined over a field of zero characteristic is reflexive. 

Given an absolutely irreducible hypersurface $\s \subset \p^n$ defined over $\fq$ by a homogeneous polynomial $F$, we define the $\fq$-Frobenius nonclassical locus of $\s$ by
$$
\s_{\Phi_q}=\{P \in \s_{\operatorname{reg}} \ : \ \Phi_q(P) \in T_P(\s)\},
$$
where $\s_{\operatorname{reg}} $ denotes the nonsingular locus of $\s$  and $T_P(S)$ denotes the tangent hyperplane to $\s$ at $P$. We say that $S$ is $\fq$-Frobenius nonclassical when $\s_{\Phi_q}=\s_{\operatorname{reg}}$. 

The Gauss map of $\s$ is the map $\Gamma$ defined in the nonsingular locus of $\s$ by $\Gamma(P)=T_P(\s)^\prime \in (\p^n)^\prime$, where $(\p^n)^\prime$ denotes the dual projective space and $H^\prime$ denotes the dual of a hyperplane $H$. The (Zariski) closure of the image of $\Gamma$ is the strict dual variety $\s^\prime$ of $\s$.  Denote by $x_0,\ldots,x_n \in \fq(\s)$ the coordinate functions of $\s$ and set
$$
F_i=\frac{\partial F}{\partial X_i}(x_0:\ldots:x_n).
$$
Since $\s:F=0$,  we have that
$$
\Gamma=(F_0:\ldots:F_n).
$$
Note that $\s$ will be $\fq$-Frobenius nonclassical if, and only if,
$$
\sum_{i=0}^{x}F_ix_i^q =0 \in \fq(\s),
$$
that is, if and only if $F$ divides $\sum\limits_{i=0}^{x}F_iX_i^q$. This happens in particular when $\sum\limits_{i=0}^{x}F_iX_i^q=0$ (as a polynomial). Under the hypothesis that $\deg(\s) \neq 0 \mod p$, the latter case was classified in \cite{ADL}. Given the above, we shall henceforth assume that $\s$ satisfies the following conditions:
\begin{enumerate}
\item[{\rm i)}] $\deg(\s) \neq 0 \mod p$;
\item[{\rm ii)}] $\sum\limits_{i=0}^{x}F_iX_i^q  \text{ is a nonzero polynomial}.$
\end{enumerate}

\section{Curves on Frobenius nonclassical loci of hypersurfaces}

This section is mainly devoted to the study of the $\fq$-Frobenius nonclassical locus of hypersurfaces and of the curves defined over $\fq$ lying on them. We start by giving a characterization of nonclassical curves in $\p^n(\kk)$, where $\kk$ is algebraically closed of characteristic $p>n$, in terms of Gauss maps of hypersurfaces.

\begin{prop}\label{class}
Let $\s \subset \p^n$ be an absolutely irreducible nondegenerate hypersurface  defined over a perfect field $\kk$ of characteristic $p>n$. 
Let $\xx$ be a curve defined over $\kk$ on $\s$ such that 
\begin{enumerate}
\item[{\rm i)}] $\xx$ is not contained in the singular locus $\operatorname{Sing}(\s)$ of $\s$;
\item[{\rm ii)}]The restriction of the Gauss map of $\s$ to $\xx$ is inseparable.
\end{enumerate}
Then, $\xx$ is nonclassical.  The converse also holds if $\kk$ is algebraically closed. More precisely, if $\kk$ is algebraically closed of characteristic $p>n$, then every nonclassical curve $\xx \subset \p^n$ defined over $\kk$ is contained in a (not necessarily unique) hypersurface $\s \subset \p^n$ with Gauss map admitting $p$-powers as coordinate functions. Also, $\xx \not\subset \operatorname{Sing}(\s)$.
\end{prop}
\begin{proof}
Consider a hypersurface $\s:F=0$ as in the statement and  let $\xx \subset \s$ be a curve satisfying conditions i) and ii). Assume that $\xx$ is defined by the coordinate functions $y_0,\ldots,y_n \in \kk(\xx)$. 
Since $\kk$ is perfect,  there exist $u_0,\ldots,u_n \in \kk(\xx)$ and $r>0$ such that $(F_0:\cdots:F_n)\vert_{\xx}=(u_0^{p^r}:\cdots:u_n^{p^r})$. In particular, $F_i(y_0,\ldots,y_n)=zu_i^{p^r}$ for some $z \in \kk(\xx)$ when we see $F_i$ as a homogeneous polynomial.
If $x_0,\ldots,x_n \in \kk(\s)$ are the coordinate functions of $\s$, we have
\begin{equation}\label{eq1}
\sum_{i=0}^{n}F_ix_i=0.
\end{equation}

 From $F(y_0,\ldots,y_n)=0$ and \eqref{eq1} we have that 
 $$
0=\sum_{i=0}^{n}F_i(y_0,\ldots,y_n)X_i(y_0,\ldots,y_n), 
 $$
 which gives
 \begin{equation}\label{eq2}
\sum_{i=0}^{n}u_i^{p^r}y_i =0\in \kk(\xx).
\end{equation}
 Since $p>n$, for every $1\leq m\leq n$ one has $D_\zeta^{(m)}(u_i^{p^{r}})=0$, where $\zeta \in \kk(\xx)$ is a separating element. Thus, from \eqref{eq2}, we obtain
 $$
 \sum_{i=0}^{n}u_i^{p^r}D_\zeta^{(m)}y_i =0, \ \ \  \text{ for } \ \ m=0,1,\ldots,n.
 $$
 Therefore, 
 $$\det\left(D_\zeta^{(m)}y_j\right)_{0 \leq m,i \leq n} = 0;$$ in other words, $\xx$ is nonclassical.  We point out that that this is essentially the ``if'' part of the proof of  \cite[Theorem 1]{GV1}) .

Conversely, assume that $\kk$ is algebraically closed and that $\xx \subset \p^n$ is a nonclassical curve defined by coordinate functions $y_0,\ldots,y_n \in \kk(\xx)$. Let $(\varepsilon_0,\ldots,\varepsilon_n)$ be the order sequence of $\xx$. By the $p$-adic criterion, see e.g. \cite[Lemma 2]{GV1}, we have that $p$ divides $\varepsilon_i$ for some $i$. In particular, we have $\varepsilon_n>n>p$. Thus $$
\det\left(D_\zeta^{(m_i)}y_j\right)_{0 \leq i,j \leq n} = 0
$$ for every separating element $\zeta \in \kk(\xx)$ and for every sequence of integers $0\leq m_0<m_1<\cdots<m_n<p$. Therefore, again by  \cite[Theorem 1]{GV1}, there exist $u_0,\ldots,u_n \in \kk(\xx)$ with $u_i \neq 0$ for at least one $i$ and $r>0$ such that \eqref{eq2} holds. Note that, after clearing the denominators, we may assume that each $u_i$ is a polynomial on the functions $y_j$. For each $i$, let $U_i \in \kk[X_0,\ldots,X_n]$ such that $u_i$ is the projection of $U_i$ in $\kk(\xx)$, and set
$$
\mathcal{V}: \sum_{i=0}^{n}U_i^{p^r}X_i=0.
$$
Since $\xx$ vanishes on the defining polynomial of $\mathcal{V}$, we have that $\xx$ is on a component $\s$ of $\mathcal{V}$, and such a component must be nonlinear as $\xx$ is nondegenerate. Further, $\xx \not\subset \operatorname{Sing}(\s)$ as $u_i \neq 0$ for some $i$. Denote by $F$ the polynomial defining $\s$, and let $L \in \kk[X_0,\ldots,X_n]$ such that $\sum\limits_{i=0}^{n}U_i^{p^r}X_i=FL$. Then,
$$
U_i^{p^r} \equiv F_i L \mod F,
$$
which gives that the Gauss map of $\s$ is $\Gamma=(U_0^{p^r}:\cdots:U_n^{p^r})$.
\end{proof}

\begin{rem}
    Let $\xx$ be an $\mathbb{F}_{q^2}$-maximal curve. Then, $\xx$ is equipped with a $g^r_{q+1}$, called the Frobenius linear series, for which $\xx$ is both nonclassical and $\mathbb{F}_{q^2}$-Frobenius nonclassical (see, for instance, \cite[Lemma 2.1 and 2.2]{KT}). The main result in \cite{KT}, usually referred to as "the Natural Embedding Theorem", states that the aforementioned $g^r_{q+1}$ gives rise to an embedding of $\xx$ into a nondegenerate Hermitian variety $\mathcal{H}$ of $\p^r(\fqc)$. Interestingly, $\mathcal{H}$ is a nonsingular, $\mathbb{F}_{q^2}$-Frobenius nonclassical hypersurface with a Gauss map admitting $p$-powers as coordinate functions. Thus, Proposition \ref{class} can be thought of as a weaker analog to the Natural Embedding Theorem for nonclassical curves. 
\end{rem}

We emphasize here that the restriction of an inseparable morphism of an algebraic variety to a subvariety is not necessarily inseparable. For instance, consider $\s=\p^2(\fqc)$ and $\xx:Y=0$. Then $\phi=(1:x:y^p):\p^2(\fqc) \lra \p^2(\fqc)$ is an inseparable morphism, but $\phi \mid_\xx$ is separable. The following example is suitable to our situation.

\begin{ex}\label{resep}
Consider $\s \subset \p^3(\fqc)$ defined over $\fq$ by $\s:X^{p+1}+Y^{p+1}+Z^2W^{p-1}+W^{p+1}=0$, where $p>3$ is the characteristic of $\fq$. Then the Gauss map of $\s$ is $\Gamma=(x^p:y^p:2z:1-z^2)$, where $x,y,z$ are the residues of $X/W,Y/W,Z/W$ in $\fq(\s)$ respectively. Note that $\fq(\s^\prime)=\fq(x^p,y^p,z)$, whence $\Gamma$ is inseparable. 

Now consider the absolutely irreducible plane curve $\cc$ defined over $\fq$ by $f(X,Y)=X^{p+1}+X^4+1+Y^{p+1}$. Then $\fq(\cc)=\fq(\xi,\zeta)$, where $f(\xi,\zeta)=0$. Let $\xx=\phi(\tilde{\cc})$, where $\phi=(\xi:\zeta:\xi^2:1)$ and $\tilde{\cc}$ is a normalization of $\cc$. Then $\xx$ is defined over $\fq$ and $\fq(\xx)=\fq(\cc)$. Note that  $\xx \subset \s \cap \mathcal{Q}$, where $\mathcal{Q}:Z=X^2$, and $\xx$ is clearly nondegenerate.  Since $\Gamma \mid_\xx=(\xi^p:\zeta^p:2\xi^2:1-\xi^4)$, we have that $\Gamma \mid_\xx$ is separable. Moreover,  $\xi$ is a separating variable of $\fq(\xx)$ and one can check directly that $D^{(3)}_\xi \zeta=-4\xi/\zeta^p$. Therefore, 
$$
\left |
\begin{array}{cccc}
1& \xi & \zeta & \xi^2 \\ [2mm]
0& 1  & D^{(1)}_\xi \zeta & 2\xi \\ [2mm]
0 & 0 & D^{(2)}_\xi \zeta & 1 \\ [2mm]
0 & 0 & D^{(3)}_\xi \zeta & 0
\end{array}
\right |=\frac{4\xi}{\zeta^p} \neq 0,
$$
that is, $\xx$ is a classical curve.
\end{ex}

In the next result, we show that the strict Gauss map of an $\fq$-Frobenius nonclassical hypersurface is inseparable. Although such a result appears in the proof of \cite[Theorem 4.5]{ADL}, in Proposition \ref{insep} below we provide an alternative proof.

\begin{prop}\label{insep}
Let $\s \subset \p^n$ be an absolutely irreducible  nondegenerate hypersurface  defined over $\fq$ such that its Gauss map is generically finite. If $\s$ is $\fq$-Frobenius nonclassical, then its Gauss map is inseparable. 
\end{prop}
\begin{proof}
Assume, on the contrary, that the Gauss map $\Gamma$ of $\s$ is separable. Since $\s$ is a hypersurface, the conormal variety of $\s$ is
$$
C(\s)=\overline{\{(P,T_P(\s)) \ :  \ P \in \s^{\operatorname{reg}}\}},
$$
and  $\Gamma=\pi^{\prime} \circ \pi^{-1}$, where $\pi:C(\s) \lra S$ and $\pi^\prime:C(\s) \lra S^\prime$ are the first and second projection, respectively. Since $\pi$ is separable, the separability of $\Gamma$ implies that $\pi^\prime$ is separable. By the Monge-Segre-Wallace criterion (see \cite{Wa}), it follows that $\s$ is reflexive, that is, $C(\s)=C(\s^\prime)$. Thus, $\s^{\prime \prime}=\s$ and $\s^\prime$ is reflexive as well. Then $\Gamma^\prime$ (the Gauss map of $S^\prime$) is separable and, since it is generically finite, it is birational. However, since $\s$ is $\fq$-Frobenius nonclassical, we have that
$$
\sum_{i=0}^{n}F_i x_i^q=0
$$
holds in the function field of $\s$. Note that $F_i \neq 0$ for all $i$, otherwise, $\s^\prime$ would be a hyperplane, and then $\s=\s^{\prime \prime}$ would be a hyperplane as well, which is a contradiction. Since $\s^\prime$ is a hypersurface (otherwise, we would not have $\s^{\prime \prime}=\s$), there exists an absolutely irreducible homogeneous polynomial $G \in \fq[T_0,\ldots,T_n]$ such that $\s^\prime:G=0$. Thus, one may see $G$ as a polynomial in $F_0,\ldots,F_n$ and the $X_i$ as quotients of homogeneous polynomials in $\fq[F_0,\ldots,F_n]$, and then
$$
\sum_{i=0}^{n}F_i X_i^q \equiv 0 \mod G. 
$$
Hence, there exist a quotient $L$ of homogeneous polynomials in $ \fq[T_0,\ldots,T_n]$ such that for all $i=0,\ldots,n$, we have 
$$
X_i^q \equiv L \cdot \frac{\partial G}{\partial F_i} \mod G. 
$$
This means that $\Gamma^\prime =(x_0^q:\cdots: x_n^q)$. Therefore, $\Gamma^{\prime}$ is purely inseparable, which is a contradiction. This finishes the proof.
\end{proof}


\begin{cor}
Let $\s \subset \p^n$ be an absolutely irreducible  hypersurface  defined over $\fq$. If $\s$ is $\fq$-Frobenius nonclassical and $\dim \s^\prime=\dim \s$, then $\s$ is nonreflexive. 
\end{cor}
\begin{proof}
From Proposition \ref{insep}, the Gauss map $\Gamma$ of $\s$ is inseparable. Since $\Gamma=\pi^{\prime} \circ \pi^{-1}$ and $\pi$ is separable, we conclude that the conormal map $\pi^{\prime}$ is inseparable. Thus, by the Monge-Segre-Wallace criterion,  $\s$ is nonreflexive.
\end{proof}

The main result of this section is the following.

\begin{thm}\label{fncl}
Let $\s \subset \p^n$ be an absolutely irreducible hypersurface  defined over $\fq$. Assume that  $p>n$ and  $\dim(\s_{\Phi_q})>0$. Then every curve $\xx \subset \s_{\Phi_q}$ defined over $\fq$ such that the restriction of the Gauss map of $\s$ to $\xx$ is inseparable is $\fq$-Frobenius nonclassical.
\end{thm}
\begin{proof}
 Set $\s:F=0$. Let $\xx \subset \s$ be a curve defined over $\fq$ and contained in $\s_{\Phi_q}$. Denote by $(\varepsilon_0,\varepsilon_1,\ldots,\varepsilon_n)$ its order sequence. Since $\s_{\Phi_q}$ does not contain singular points of $\s$, $\xx$  is nonclassical by Proposition \ref{class}. If $\varepsilon_{n-1}>n-1$, then $\xx$ is $\fq$-Frobenius nonclassical, since the $\fq$-Frobenius order-sequence is a subsequence of the order sequence. Assume that $\varepsilon_{n-1}=n-1$; in particular, $\varepsilon_n>n$. Then, if   $x_0,\ldots,x_n \in \fq(\s)$ denote the coordinate functions of $\s$, relation \eqref{eq1} holds with all $F_i$ defined over $\fq$. As in the proof of Proposition \ref{class}, if $y_0,\ldots,y_n \in \fq(\xx)$ are the coordinate functions of $\xx$, then \eqref{eq2} holds as well, where $\Gamma \mid_\xx=(u_0^{p^r}:\cdots:u_n^{p^r})$. Note that, up to the replacement of $r$ by a bigger number, we may assume that at least one of the $u_i$ is a separating element of $\fq(\xx)$.  From \cite[Theorem 7.65]{HKT}, we have $\varepsilon_n=p^r$ and the osculating hyperplane at a general point $P \in \xx$ is given by
$$
H_P(\xx):\sum_{i=0}^{n}\left(u_i^{p^r}(P) \right)X_i=0.
$$
However, from \eqref{eq1}, $H_P(\xx)=T_P(\s)$ for a general $P \in \xx$, that is, the osculating hyperplane to $\xx$ at $P$ coincides with the tangent hyperplane to $\s$ at $P$.  Since  $\Phi_q(P) \in T_P(\s)$ for a general point $P \in \s_{\Phi_q}$, we have that $\xx \subset \s_{\Phi_q}$ implies that $\Phi_q(P) \in T_P(\s)=H_P(\xx)$ for a general point $P \in \xx$. Therefore, by \cite[Corollary 1.3 and Proposition 2.1]{SV}, $\xx$ is $\fq$-Frobenius nonclassical.
\end{proof}

Theorem \ref{fncl} then immediately yields the following result.

\begin{thm}\label{fnc}
Let $\s \subset \p^n$ be an absolutely irreducible $\fq$-Frobenius nonclassical hypersurface. If $p>n$, then every curve  $\xx \subset \s$ defined over $\fq$ not contained in the singular locus of $\s$ such that the restriction of the Gauss map of $S$ to $\xx$ is inseparable is $\fq$-Frobenius nonclassical.
\end{thm}

We now provide a number of examples that highlight the different implications of the results presented so far.

\begin{ex}
Let $\s \subset \p^3$ defined over $\fq$ by $F(W,X,Y,Z)=0$, where $F(W,X,Y,Z)=(Z-W)^qW-X^qY-Y^qX+ZW^q$ and $p>3$. Then one can check that $\s$ is $\fqs$-Frobenius nonclassical with strict Gauss map $\Gamma=(-Y^q:-X^q:W^q:(Z-W)^q)$. By Theorem \ref{fnc}, given a rational map $\psi:\p^2 \dashrightarrow \p^3 $, any irreducible nondegenerate component of $\overline{\psi(\p^2)}\cap \s$ defined over $\fqs$ is an $\fqs$-Frobenius nonclassical curve. For instance, if $\psi=(W^2: WX:WY:XY)$, then   $\overline{\psi(\p^2)}\cap \s$ has a nondegenerate component $\xx$ that is $\fq$-birationally equivalent to the Artin-Mumford curve defined by the affine equation $\mathcal{M}:(x^q-x)(y^q-y)=1$. Further,  such a model $\xx \subset \p^3$ of the Artin-Mumford curve is $\fqs$-Frobenius nonclassical, as presented in \cite[Subsection 5.1]{Ar2}. It should be noted that the number of $\fqs$-rational points of $\xx$ is $N_2=q^2(q-1)+2q$, its order sequence is $(0,1,2,q)$ and its degree is $2q$,  see \cite[Subsection 5.1]{Ar2}. The St\"ohr-Voloch bound for $\fqs$-Frobenius classical curves in $\p^3$ with these parameters is $N_2 \leq 2(q^2+3q-3)$. Therefore, the number of $\fqs$-rational points of $\xx$ exceeds this bound by far.
\end{ex}

\begin{ex}\label{fermat}
Set $n=(q^\ell-1)/(q-1)$ and $M={s+2 \choose 2}-1$ for some $\ell>1$, $s \in \{2,\ldots,n-1\}$ and  consider the Fermat hypersurface $\s_n \subset \p^M$ defined over $\fq$ with $\operatorname{char}(\fq)=p>M$ by $F_s=0$, where
$$
F_s=\sum_{i=0}^{M}a_i X_i^n.
$$ 
Then $\frac{\partial F_s}{\partial X_i}=(na_i X_i^{(q^{\ell-1}-1)/(q-1)})^q$, and 
$$
n\sum_{i=0}^{M}a_i X_i^n=\sum_{i=0}^{M}\left(na_i X_i^{\frac{q^{\ell-1}-1}{q-1}}\right)^{q}X_i=\sum_{i=0}^{M}\frac{\partial F_s}{\partial X_i} X_i.
$$ 
Therefore,
$$
\sum_{i=0}^{M}\frac{\partial F_s}{\partial X_i} X_i^{q^\ell}=n\left(\sum_{i=0}^{M}a_i X_i^{\frac{q^\ell-1}{q-1}}\right)^q=n\left(\sum_{i=0}^{M}a_i X_i^{n}\right)^q \equiv 0 \mod F_s,
$$
that is, $\s_n$ is $\F_{q^\ell}$-Frobenius nonclassical. Consider the Veronese embedding
$$
\phi_s=(Z^s: XZ^{s-1}:YZ^{s-1}:X^2Z^{s-2}: \cdots:X^iY^jZ^k:\cdots:Y^s):\p^2\hookrightarrow \p^M,
$$
where $i+j+k=s$. When the nondegenerate curve $\phi_s(\p^2)\cap \s_n$ is irreducible, then it corresponds to the $\F_{q^\ell}$-Frobenius nonclassical curve with respect to the linear system of all plane curves of degree $s$ appearing in \cite{AB}.
\end{ex}

The following example shows that we can have Frobenius nonclassical curves on Frobenius classical surfaces.

\begin{ex}\label{exfermat}
Let $\cc$ be the plane curve defined over $\F_5$ by the affine equation
$$
(x+y)^5y^5x-y(1+xy)^5-x^5(1+xy)(x+y)^{4}+(x+y)^5=0.
$$
One can check via Magma \cite{MAGMA} that $\cc$ is absolutely irreducible. Moreover, the point $P=(-1:0:1) \in \cc$ is such that the intersection multiplicity of $\cc$ with the tangent line to $\cc$ at $P$ is $2$, whence $\cc$ is a classical curve. Set
$$
z=\frac{-1-xy}{x+y} \in \F_5(x,y).
$$
Then we have 
\begin{equation}\label{inters1}
xy+yz+xz+1=0  
\end{equation}
and
\begin{equation}\label{inters2}
 xy^5+yz^5+x^5z+1=0. 
\end{equation}

If $\xx=\varphi(\tilde{\cc}) \subset \p^3(\overline{\F}_ 5)$ , where $\tilde{\cc}$ is a normalization of $\cc$ and $\varphi=(1:x:y:z)$, we have that $\xx$ is a model of $\cc$ (that is, $\xx$ is $\F_5$-birationally equivalent to $\tilde{C}$) and $\xx \subset \s$, where
$$
\s:XY^5+YZ^5+X^5Z+W^{6}=0.
$$
Note that $\s$ is smooth and its Gauss map is $\Gamma=(Y^5:Z^5:X^5:W^5)$, which is purely inseparable (here, the coordinates of $\Gamma$ are seen as functions in $\F_5(\s)$). Moreover,  $\s$ is $\F_5$-Frobenius classical, since $\s_{\Phi_5}$ is the intersection of $\s$ with the quadric surface
$$
\mathcal{Q}: XY+YZ+XZ+W^2=0.
$$
As a matter of fact, from \eqref{inters1} and \eqref{inters2}, we conclude that  $\xx\subset \s_{\Phi_5}$ . Equation \eqref{inters1} implies that $\xx$ is nondegenerate. Now we compute the order sequence of $\xx$. It is known that $\varepsilon_0=0$ and $\varepsilon_1=1$. We claim that $\varepsilon_2=2$. In fact, it can be checked that $x$ is a separating element in $\fq(\xx)$. If $\varepsilon_2>2$, equation  \eqref{inters2} gives that $D_x^{(2)}y=D_x^{(2)}z=0$.  However, if $D_x^{(2)}y=0$, then the plane curve $\cc$ is nonclassical, which is a contradiction. Hence $\varepsilon_2=2$. Finally, \eqref{inters2} implies that $\varepsilon_3=5$, and   \eqref{inters1} implies that $\Phi_5(P) \in H_P(\xx)=T_P(\s)$ for a general $P \in \xx$, which gives that the $\F_5$-Frobenius order sequence of $\xx$ is $(0,1,5)$, that is, $\xx$ is $\F_5$-Frobenius nonclassical.
\end{ex}

The next example illustrates how Frobenius nonclassical curves are a source of curves with many rational points.

\begin{ex}
Consider the curve $\ff$ defined over $\fq$, with $q=p^h$, by the affine equation $x^ny^n-x^n-y^n=1$. If $\varphi=(1:x:y:xy):\ff \dashrightarrow \p^3$, we have that $\xx:=\overline{\varphi(\ff)}$ is $\fq$-birationally equivalent to $\ff$ and lies on the smooth Fermat surface $\s$ defined by
$$
-X^n-Y^n+Z^n-W^n=0.
$$
If $\xx$ is $\fq$-Frobenius classical, then by \cite[Corollary 3.3]{BC} the number $N_q$ of the rational points on a nonsingular model of $\xx$ is bounded by
\begin{equation}\label{cotaleandro}
N_q \leq 2(n^2-2n)+\frac{2n(q+3)}{3}-\frac{2(n-2)(2n+r+s)}{3},
\end{equation}
where $r$ and $s$ are, respectively, the number of roots of $T^n-1$ and $T^n+1$ in $\fq$. Assume that $n=(q-1)/(p^r-1)$, where $r \mid h$. Then we obtain that $\s$ is $\fq$-Frobenius nonclassical with Gauss map 
$$
\Gamma= \left((-X^{\frac{p^{h-r}-1}{p^r-1}})^{p^r}: (-Y^{\frac{p^{h-r}-1}{p^r-1}})^{p^r}:(Z^{\frac{p^{h-r}-1}{p^r-1}})^{p^r}:(-W^{\frac{p^{h-r}-1}{p^r-1}})^{p^r}\right).
$$
Thus, $\Gamma \mid_{\xx}$ is inseparable. Hence, in this case $\xx$ is $\fq$-Frobenius nonclassical by Theorem \ref{fnc}. One can check that $D_x^{(2)}y \neq 0$, whence we deduce that the $\fq$-Frobenius order sequence of $\xx$ is $(0,1,p^r)$. Now, for $n=(q-1)/(p^r-1)$ we have
\begin{equation}\label{tl}
N_q=n^2(p^r-3)+4n,
\end{equation}
see \cite[Teorema 3.5.1]{Ro}. Note that \eqref{tl} exceeds bound \eqref{cotaleandro} for sufficiently large values of $q$. For instance, let $\ff:x^8y^8-x^8-y^8=1$ defined over $\F_{49}$. In this case we have $r=s=8$ and the right side of \eqref{cotaleandro} is $<245$. However, the exact number of $\F_{49}$ rational points of $\ff$ is $N_{49}=288$.
\end{ex}

\begin{rem}
In \cite{BN}, the authors provide a nice bound for the number of $\fq$-rational points on curves lying on $\fq$-Frobenius nonclassical surfaces in $\p^3$ (\cite[Theorem 2.5]{BN}), which can improve some known bounds in certain cases. Later, in \cite[Proposition 3.2]{BN} and in the paragraph immediately following it, it is claimed that the $\fq$-Frobenius nonclassical locus of surfaces in $\p^3$ cannot contain nondegenerate $\fq$-Frobenius nonclassical curves with $\nu_1=1$. However, there appears to have been an oversight in this result. Note that all the examples in this section contradict this statement.
\end{rem}

We end this section with a generalization of the idea behind Example \ref{fermat}. This idea  can be used to construct Frobenius nonclassical plane curves with respect to the linear systems of all plane curves of a given degree. As explained in \cite{AB}, although such curves are somewhat rare, their construction is desirable as it provides a potential source of curves with many points, see \cite{AB1,GV2,Gi}; and also \cite[Section 8.7]{HKT}. 

Let $s>0$ be an integer and set $M={s+2 \choose 2}-1$. Again, consider the Veronese embedding
$$
\phi_s=(Z^s: XZ^{s-1}:YZ^{s-1}:X^2Z^{s-2}: \cdots:X^iY^jZ^k:\cdots:Y^s):\p^2\hookrightarrow \p^M,
$$
where $i+j+k=s$. Since $\phi_s:\p^2 \lra \phi_s(\p^2)$ is an isomorphism of algebraic varieties, there exists a bijective correspondence between plane irreducible projective curves and irreducible projective curves lying on $ \phi_s(\p^2)$.  

\begin{thm}\label{wrts}
With notation as above, let $p>M$ and let $\s \subset \p^M(\fqc)$ be an absolutely irreducible $\fq$-Frobenius nonclassical hypersurface defined by $F(X_0,\ldots,X_M)=0$ with Gauss map $\Gamma$.  Assume that there exists an absolutely irreducible component $\xx$ of $\s \cap \phi_s(\p^2)$ defined over $\fq$ such that
\begin{itemize}
\item $\xx$ is nondegenerate;
\item $\xx$ is not contained in the singular locus of $\s$;
\item $\Gamma\mid_{\xx}$ is inseparable.
\end{itemize}
Set $\ff:=\phi_s^{-1}(\xx) \subset \p^2(\fqc)$. Then
\begin{itemize}
\item [(i)] $\ff$ is defined over $\fq$ by an absolutely irreducible factor of
$$
F(Z^s, XZ^{s-1},YZ^{s-1},X^2Z^{s-2}, \ldots,X^iY^jZ^k, \dots,Y^s) \in \fq[X,Y,Z];
$$
\item [(ii)] If $\deg{\ff}>s$, then $\ff$ is $\fq$-Frobenius nonclassical with respect to the linear system of all plane curves of degree $s$.
\end{itemize}
\end{thm}
\begin{proof}
Let $\ff=\phi_s^{-1}(\xx)$ as in the statement. Since the Frobenius map $\Phi_q$ commutes with $\phi_s$, we have that $\ff$ is defined over $\fq$. The irreducibility of $\ff$ follows from the fact that $\phi_s$ is an isomorphism, and hence it preserves irreducible algebraic subsets. Further, if $\fq(\ff)=\fq(x,y)$ with $h(x,y)=0$ for some absolutely irreducible $h \in \fq[X,Y]$ (i.e., $1,x,y$ are affine coordinate functions of $\ff$ ), we have that $\phi_s(1:x:y) \in \s$, which gives
$$
F(1, x,y,x^2, \ldots,x^iy^j, \dots,y^s)=0,
$$
proving (i).

It is well known that the linear system of all plane curves of degree $s<n:=\deg{\ff}$ cuts out on $\ff$ a base point free linear series $g_{ns}^{M}$, see e.g. \cite[Chapters 6 and 7 ]{HKT}. The model in $\p^M$ arising from this $g_{ns}^{M}$ is given by the morphism $\phi_s\mid_\ff=(1: x:y:x^2: \cdots:x^iy^j: \cdots : y^s)$.  Hence, such a model is  $\phi_s\mid_\ff(\ff)=\xx$. The intersection divisors of $\ff$ with the plane curves of degree $s$ are precisely the intersection divisors of $\xx$ with hyperplanes of $\p^M$. In particular, $\ff$ is $\fq$-Frobenius nonclassical with respect to the linear system of plane curves of degree $s$ if, and only if, $\xx$ is $\fq$-Frobenius nonclassical. Therefore, the result follows from Theorem \ref{fnc}.
\end{proof}

An immediate consequence of Theorem \ref{wrts} is the following.

\begin{cor}
With notation as in Theorem \ref{wrts}, let $p>M$ and  let $\s \subset \p^M(\fqc)$ be an absolutely irreducible smooth $\fq$-Frobenius nonclassical hypersurface defined by the equation $F(X_0,\ldots,X_M)=0$ with Gauss map $\Gamma$. Assume that $\xx:=\s \cap \phi_s(\p^2)$ is an absolutely irreducible nondegenerate curve such that $\Gamma \mid_\xx$ is inseparable. Then the plane curve
$$
\ff:F(Z^s, XZ^{s-1},YZ^{s-1},X^2Z^{s-2}, \ldots,X^iY^jZ^k, \dots,Y^s) =0
$$
is $\fq$-Frobenius nonclassical with respect to the linear system of all plane curves of degree $s$.
\end{cor}
\section{Certain Frobenius nonclassical hypersurfaces with separated variables }

Following the definition given in \cite{ADL2}, a hypersurface $\s \subset \p^n$ defined over a field $\kk$ is said to have separated variables if, up to a projective change of coordinates over $\kk$, it is defined by $F(X_0,\ldots,X_n)=0$, where
$$
F(X_0,\ldots,X_n)=G(X_0,\ldots,X_m)+H(X_{m+1},\ldots,X_n)
$$
for some $m \in \{0,\ldots,n-1\}$. 

In this section, we will investigate $\fq$-Frobenius nonclassical hypersurfaces $\s$ with separated variables defined over $\fq$ by polynomials of type
$$
F(X_0,\ldots,X_n)=G(X_0,\ldots,X_{n-1})-X_n^d.
$$
In other words, if $d\mid q-1$, then $\s$ can be regarded as a Kummer cover of the projective space $\p^{n-1}$.

\begin{thm}\label{kummer}
Let $\s:F(X_0,\ldots,X_n)=0 \subset \p^n$ be a smooth absolutely irreducible hypersurface of degree $d>2$, where $F(X_0,\ldots,X_n) \in \fq[X_0,\ldots,X_n]$ is a homogeneous polynomial such that
$$
F(X_0,\ldots,X_n)=G(X_0,\ldots,X_{n-1})-X_n^d.
$$
Then, $\s$ is $\fq$-Frobenius nonclassical if, and only if, there exists a positive integer $k$ such that  $kd=d-1+q$ with $p \mid k$ and 
\begin{equation}\label{n-1}
G^k=\sum_{i=0}^{n-1}G_i X_i^q,
\end{equation}
where $G_i=\partial G/\partial X_i$. In particular, $\mathcal{V}:G=0 \subset \p^{n-1}$ is $\fq$-Frobenius nonclassical as well. 
\end{thm}
\begin{proof}
Assume that $\s$ is $\fq$-Frobenius nonclassical and set $F_i=\partial F/\partial X_i$. Then $F_i=G_i$ for $i=0,\ldots,n-1$ and $F_n=-dX_n^{d-1}$.  By \cite[Theorem 1.9]{ADL2} we conclude that $d \equiv 1 \mod p$, as $\s$ is smooth. Since $\s$ is $\fq$-Frobenius nonclassical, we have that 
$$
F=G-X_n^d  \ \text{ divides }\ \sum_{i=0}^{n-1}G_i X_i^q-X_n^{d-1+q}.
$$
Using \cite[Lemma 3.1]{Bo}, we have that there exists a polynomial $R \in \fq(X_0,\ldots,X_{n-2})[T]$ such that
$$
R(G(X_0,\ldots,X_{n-1}))=\sum_{i=0}^{n-1}G_i X_i^q \ \ \text{ and } \ \ R(X_n^d)=X_n^{d-1+q}.
$$ 
Thus, $R(T)=T^k$ for some integer $k>0$ and $kd=d-1+q$. Since $p \mid d-1$ and $p \nmid d$, we conclude that $p \mid k$. The converse is straightforward.
\end{proof}

\begin{cor}\label{kummer2}
Let $\s:F(X_0,\ldots,X_n)=0 \subset \p^n$ be a smooth absolutely irreducible hypersurface of degree $d>2$, where $F(X_0,\ldots,X_n) \in \fq[X_0,\ldots,X_n]$ is a homogeneous polynomial such that
$$
F(X_0,\ldots,X_n)=G(X_0,\ldots,X_{n-1})-X_n^d.
$$
If $\s$ is $\fq$-Frobenius nonclassical, then every curve $\xx \subset \s$ defined over $\fq$ is $\fq$-Frobenius nonclassical.
\end{cor}
\begin{proof}
By Theorem \ref{kummer}, the Gauss map of $\s$ is $\Gamma=(\alpha_0:\cdots:\alpha_{n-1}:x_n^{d-1})$, where $\alpha_0,\ldots,\alpha_{n-1},x_n$ are respectively the residues of $G_i$  and $X_n$ in $\fq(\s)$, $i=0,\ldots,n-1$. By using the fact that $p \mid k$ and $d-1<q$, and by comparing the exponents of each variable in both sides of \eqref{n-1} we have that each  $G_i$ is a $p$-power, hence, so is $\alpha_i$. Also, $p \mid d-1$. Hence, given $\xx \subset \s$, we have that $\Gamma\mid_\xx$ is inseparable. The result then follows from Theorem \ref{fnc}.
\end{proof}

\begin{rem}\label{obsutil}
 In Theorem \ref{kummer}, the hypothesis that $\s$ is smooth is used only to ensure that $p \mid d-1$. It is therefore possible that this result can be extended to singular hypersurfaces.
\end{rem}

For the rest of this section, set $\mathcal{V}_i:X_i=0$. Another interesting case where we can investigate the Frobenius nonclassicality of a hypersurface given by separated variables is the following. Let $\s$ be defined by
\begin{equation}\label{hfg}
F(X_0,\ldots,X_n)=G(X_0,X_1,X_2)+a_3X_3^d+\cdots+a_nX_n^d.
\end{equation}

\begin{prop}\label{sepprot}
Let $\s:F(X_0,\ldots,X_n)=0 \subset \p^n$ be an absolutely irreducible hypersurface of degree $d>2$, where $F(X_0,\ldots,X_n) \in \fq[X_0,\ldots,X_n]$ is as in \eqref{hfg}, and $q$ is a power of a prime $p>3$. 
Assume that $\mathcal{V}:=\s\cap \mathcal{V}_3\cap \cdots\cap \mathcal{V}_n \subset \s_{\operatorname{reg}}$. If $\s$ is $\fq$-Frobenius nonclassical, then $p \mid d-1$, and there exists an integer $r>0$  and homogeneous polynomials of degree $(d-1)/p^r$ $A(X_0,X_1,X_2)$, $B(X_0,X_1,X_2)$, $C(X_0,X_1,X_2) \in \fq[X_0,X_1,X_2]$ such that
$$
G(X_0,X_1,X_2)=A(X_0,X_1,X_2)^{p^r}X_0+B(X_0,X_1,X_2)^{p^r}X_1+C(X_0,X_1,X_2)^{p^r}X_2.
$$
\end{prop}
\begin{proof}
Assume that $\s$ is $\fq$-Frobenius nonclassical and consider the plane curve $\ff:G(X_0,X_1,X_2)=0$ defined over $\fq$. Then $\ff$ is smooth, and in particular, it is absolutely irreducible. Indeed, since $\partial F/\partial X_i=\partial G/\partial X_i$ for $i=0,1,2$ and $\partial F/\partial X_i=a_idX_i^{d-1}$ for $i \geq 3$, if $P=(a:b:c)$ is a singular point of $\ff$, then $Q=(a:b:c:0:\cdots:0)$ is a singular point of $\mathcal{V}$, contradicting the smoothness of $\mathcal{V}$. Set $F_i=\partial F/\partial X_i$ and $G_i=\partial G/\partial X_i$. The $\fq$-Frobenius nonclassicality of $\s$ gives that
\begin{equation}\label{eqfnc3}
F(X_0,\ldots,X_n) \text{ divides }  \sum_{i=0}^{n} F_iX_i^q.
\end{equation}
From $G(X_0,X_1,X_2)=F(X_0,X_1,X_2,0,\ldots,0)$ and $\partial F/\partial X_i=\partial G/\partial X_i$ for $i=0,1,2$, \eqref{eqfnc3} implies
\begin{equation}\label{eqfnc2}
G(X_0,X_1,X_2) \text{ divides } G_0X_0^q+G_1X_1^q+G_2X_2^q.
\end{equation}
In turn,  \eqref{eqfnc2} is equivalent to $\ff$ being $\fq$-Frobenius nonclassical. Since $p>3$, \cite[Proposition 2.1]{SV} implies that $\ff$ is nonclassical. Hence, once $\ff$ is smooth, we conclude from a result by  Homma \cite[Corollary 2.5]{Ho} that there exist homogeneous polynomials $A(X_0,X_1,X_2)$, $B(X_0,X_1,X_2)$, $C(X_0,X_1,X_2) \in \fq[X_0,X_1,X_2]$  and an integer $r>0$ such that
$$
G(X_0,X_1,X_2)=A(X_0,X_1,X_2)^{p^r}X_0+B(X_0,X_1,X_2)^{p^r}X_1+C(X_0,X_1,X_2)^{p^r}X_2. 
$$
Also, $A,B$ and $C$ have degree $(d-1)/p^r$. In particular, $d \equiv 1 \mod p^r$. 
\end{proof}

\begin{cor}\label{sepres}
Let $\s:F(X_0,\ldots,X_n)=0 \subset \p^n$ be a smooth absolutely irreducible hypersurface of degree $d>2$, where $F(X_0,\ldots,X_n) \in \fq[X_0,\ldots,X_n]$ is a homogeneous polynomial such that
$$
F(X_0,\ldots,X_n)=G(X_0,X_1,X_2)+a_3X_3^d+\cdots+a_nX_n^d.
$$
Assume that $p>3$ and that $\s$ is $\fq$-Frobenius nonclassical. Then every curve $\xx \subset \s$ defined over $\fq$ is $\fq$-Frobenius nonclassical. 
\end{cor}
\begin{proof}
The proof of this statement is analogous to the proof of Corollary \ref{kummer2}.
\end{proof}

We now present a characterization of hypersurfaces defined by \eqref{hfg}, whenever $n>3$.

\begin{thm}\label{46}
For $n>3$, let $\s:F(X_0,\ldots,X_n)=0 \subset \p^n$ be an absolutely irreducible hypersurface of degree $d>2$, where $F(X_0,\ldots,X_n) \in \fq[X_0,\ldots,X_n]$ is a homogeneous polynomial such that
\begin{equation}\label{nova}
F(X_0,\ldots,X_n)=G(X_0,X_1,X_2)+a_3X_3^d+\cdots+a_nX_n^d,
\end{equation}
with $a_i \neq 0$ for all $i= 1,\ldots,n$  and $q$ is a power of a prime $p>3$. Assume that $\mathcal{V}:=\s\cap \mathcal{V}_3\cap \cdots\cap \mathcal{V}_n \subset \s_{\operatorname{reg}}$.  Then $\s$ is $\fq$-Frobenius nonclassical if, and only if, there exists $r>0$ such that $d=\frac{q-1}{p^r-1}$, and homogeneous polynomials of degree $(d-1)/p^r$ $A(X_0,X_1,X_2)$, $B(X_0,X_1,X_2)$, $C(X_0,X_1,X_2) \in \fq[X_0,X_1,X_2]$ such that
\begin{eqnarray}
G(X_0,X_1,X_2)&=&A(X_0,X_1,X_2)^{p^r}X_0+B(X_0,X_1,X_2)^{p^r}X_1+C(X_0,X_1,X_2)^{p^r}X_2 \nonumber \\
                          &=&A(X_0,X_1,X_2)X_0^{q/p^r}+B(X_0,X_1,X_2)X_1^{q/p^r}+C(X_0,X_1,X_2)X_2^{q/p^r} \nonumber
\end{eqnarray}
and $a_i/a_n \in \F_{p^r}$ for all $i=3,\ldots,n-1$.
\end{thm}
\begin{proof}
 First, after dividing by $a_n$ both sides of \eqref{nova}, we may assume that $a_n=1$. Now, let $\s$ be $\fq$-Frobenius nonclassical. From Proposition \ref{sepprot} there exist an integer $r>0$ and polynomials $A,B,C$ satisfying the equality above; also, $p \mid d-1$. Note that $r$ can be chosen so that at least one of the $A,B,C$ is not a $p$-power. Remark \ref{obsutil} ensures that  the statements of Theorem \ref{kummer} hold for $S$.  In particular, $d=\frac{q-1}{k-1}$ with $k \equiv 0 \mod p$ and 
\begin{equation}\label{sl1}
\begin{split}
(-G(X_0,X_1,X_2)-a_3X_3^d-\cdots-a_{n-1}X_{n-1}^d)^k=\ & -A^{p^r}X_0^q-B^{p^r}X_1^q-C^{p^r}X_2^q \\
& -a_3X_3^{dk}-\cdots  -a_{n-1}X_{n-1}^{dk}.
\end{split}
\end{equation}
Observe that $d=p^rm+1$ for $m=(d-1)/p^r>0$ and $p^r<q$ (since $d<q$). Let $p^\ell$ be the highest power of $p$ dividing $k$. From \eqref{sl1}, we must have $\ell\leq r$ since, otherwise, $A,B,C$ will all be $p^\ell$ powers. On the other hand, $dk=d-1+q=p^rm+q=p^r(m+q/p^r)$. Thus $p^r \mid k$, which gives $r=\ell$. Hence, $k=p^rc$ for some $c>0$ co-prime with $p$, and \eqref{sl1} provides
\begin{equation}
\begin{split}
(-G(X_0,X_1,X_2)-a_3X_3^d-\cdots-a_{n-1}X_{n-1}^d)^c=\ &-AX_0^{q/p^r}-BX_1^{q/p^r}-CX_2^{q/p^r} \\
& -a_3^{1/p^r}X_3^{dc}-\cdots-a_{n-1}^{1/p^r}X_{n-1}^{dc}.
\end{split}
\end{equation}
Since $A,B,C \in \fq[X_0,X_1,X_2]$, equality above implies $c=1$ and $a_i^{p^r}=a_i$ for all $i=3,\ldots,n$. Finally, \eqref{sl1} gives 
$$
G(X_0,X_1,X_2)^{p^r}=A^{p^r}X_0^q+B^{p^r}X_1^q+C^{p^r}X_2^q,
$$
proving the last assertion. The converse is straightforward.  
\end{proof}


The Fermat hypersurfaces $\s_n \subset \p^M$ presented in Example \ref{fermat} illustrate both Proposition \ref{sepprot} and Corollary \ref{sepres}. The following example also pertains to both results.

\begin{ex}\label{garcia}
Assume $p>3$ and consider the smooth plane curve $\ff$ defined over $\F_{p^3}$ by the equation $Y^{p^2+p+1}=(X^{p+1}+Z^{p+1})^pX+(X^p+XZ^{p-1})^pZ^{p+1}$. The curve $\ff$ is $\F_{p^3}$-Frobenius nonclassical, see \cite[Section 3]{Ga}. Let $\s \subset \p^3$ be the surface defined by the equation $F=0$, where
$$
F=(X^{p+1}+Z^{p+1})^pX+(X^p+XZ^{p-1})^pZ^{p+1}+Y^{p^2+p+1}+W^{p^2+p+1}.
$$
One can check that $\s$ is smooth. Moreover, $\s$ is $\F_{p^3}$-Frobenius nonclassical. Indeed, $F_XX^{p^3}+F_YY^{p^3}+F_ZZ^{p^3}+F_WW^{p^3}$
\begin{eqnarray}
& =&(X^{p+1}+Z^{p+1})^pX^{p^3}+Y^{p^2+p}Y^{p^3}+(X^p+XZ^{p-1})^pZ^{p}Z^{p^3}+W^{p^2+p}W^{p^3} \nonumber \\\
&=&\left( (X^{p+1}+Z^{p+1})X^{p^2}+Y^{p^2+p+1}+(X^p+XZ^{p-1})Z^{p^2+1}+W^{p^2+p+1}\right)^p=F^p \nonumber 
\end{eqnarray}
By Corollary \ref{sepres}, the curves defined over $\F_{p^3}$ lying on $\s$ are a source of  $\F_{p^3}$-Frobenius nonclassical curves in $\p^3$.
\end{ex}

It should be noted that one can also combine Theorem \ref{wrts} with Theorem \ref{46} to construct new Frobenius nonclassical plane curves with respect to the linear system of all plane curves of a given degree. 


\section{Properties of a class of Frobenius nonclassical hypersurfaces}

In this section, we will investigate certain Frobenius nonclassical hypersurfaces defined over $\fq$. More precisely, we will assume that $\s \subset \p^n$ is an $\fq$-Frobenius nonclassical hypersurface satisfying the following property:
\begin{itemize}
\item [(A)] The Gauss map $\Gamma$ of $\s$ admits coordinates of the form $\Gamma=(U_0^{p^r}:\cdots:U_n^{p^r})$ with $r>0$, where $(U_0:\ldots:U_n)$ is a generically finite and separable rational map. 
\end{itemize}

As far as the authors know, all known examples of Frobenius nonclassical hypersurfaces with generically finite Gauss map satisfy  property (A) above. Our first result provides a bound for the inseparable degree of $\Gamma$.

\begin{prop}
Let $\s \subset \p^n$ be an absolutely irreducible $\fq$-Frobenius nonclassical hypersurface  defined over $\fq$ satisfying (A). Then the inseparable degree $\deg_i(\Gamma)$ satisfies
$$
p^{n-1} \leq \deg_i(\Gamma) \leq q^{n-1}.
$$
\end{prop}
\begin{proof}
Since $\s$ satisfies property (A), $\deg_i(\Gamma)=(p^{r})^{n-1}$. In particular, $p^{n-1} \leq \deg_i(\Gamma)$. Set $q=p^h$ and suppose  that $q^{n-1}<  \deg_i(\Gamma)$. Since $F_i=ZU_i^{p^r}$ for some $Z \in \fq(S)$ for $i=0,\ldots,n$, then
$$
0=\sum_{i=0}^{n}F_i x_i^q= Z \cdot \sum_{i=0}^{n}U_i^{p^r} x_i^{p^h}= Z \cdot \left(\sum_{i=0}^{n}U_i^{p^{r-h}} x_i\right)^{p^h} \Lra \sum_{i=0}^{n}U_i^{p^{r-h}} x_i=0.
$$
Proceeding as in the proof of Proposition \ref{class}, the last equality implies that $\Gamma=(U_0^{p^{r-h}}:\cdots:U_n^{p^{r-h}})$. Now, the inseparable degree of this map is $(p^{r-h})^{n-1}$ since $(U_0:\ldots:U_n)$ is generically finite and separable. This gives $h =0$, a contradiction.
\end{proof}

The assumption that $\Gamma$ is generically finite imposes some relevant restrictions whenever $n>2$. For instance, let $\s \subset \p^n$ with $n>2$ be an absolutely irreducible hypersurface defined over a perfect field $\kk$ of positive characteristic $p>0$, and let $x_0,x_1,\ldots,x_n \in \kk(\s)$ be the projective coordinate functions of $\s$;  if $\Gamma$ is generically finite, then  no more than one of the $x_i$ is a $p$-th power in $\kk(\s)$. In fact, in a suitable open set of $\s$, we may assume that $x_0=1$ and $\kk(\s)=\kk(x_1,\ldots,x_n)$. Without loss of generality,  assume that $x_n \neq 0$ is a  $p$-th power in $\kk(\s)$. Then, there exists $w \in \kk(\s)$ such that $x_n=w^p$. If $F(x_1,\cdots,x_n)=0$ is the affine equation defining $\s$, there exist $U,H,G \in \kk[X_1,\ldots,X_n]\backslash{\{0\}}$, with neither $U$ nor $G$ divisible by $F$, such that
\begin{equation}\label{powp}
G^pX_n-H^p=UF^m
\end{equation}
for some $m>0$. 
Also, $p \nmid m$. Indeed, if $p \mid m$, then \eqref{powp} gives $G^p=U_nF^m$, and then $F \mid G$, a contradiction. Now, for $j \neq n$, we have
$$
0=F^{m-1}(mUF_j+U_jF) .   
$$
Hence,
$$
U_jF=-mUF_j.
$$
Since $F$ does not divide both $F_j$ and $U$, we must have $U_j=F_j=0$ for all $j \neq n$. This means that the Gauss map of $\s$ is locally given by $\Gamma=(1:0:\cdots:0:f_n)$, where $f_n$ is the residue of $F_n$ in $\kk(S)$. In particular, the function field of $\s^\prime$ is $\kk(\s^\prime)=\kk(f_n)$. Since $n>2$, the extension
$$
\kk(\s)=\kk(x_1,\ldots,x_n)/\kk(f_n)=\kk(\s^\prime)
$$
is transcendent. Therefore, $\Gamma$ is not generically finite.

Also, if $\s$ is a cone, then $\Gamma$ is not generically finite. To see this, without loss of generality, consider $\s \subset \p^n$ defined over a field $K$ (of arbitrary characteristic) by a homogeneous polynomial $F(X_0,\ldots,X_{n-1})$. Then $F_n=0$ and $F_i$ does not depend on the variable $X_n$. Hence, for a general point $P=(a_0:\cdots:a_n) \in \s$, one has $\Gamma^{-1}\{\Gamma(P)\}\supset\{(a_0:a_1:\cdots:a_{n-1}:t) \ : \ t \in K\}$, i. e., the general fiber of $\Gamma$ contains a line. 

When $\s$ is an $\fq$-Frobenius nonclassical plane curve, by a classical result, then  $\Gamma$ is purely inseparable, $\s=\s^{\prime \prime}$ and $\Gamma^\prime \circ \Gamma=\Phi_q$. However, this no longer holds for higher dimensional hypersurfaces. For example, consider the surface $\s \subset \p^3$ defined over a finite field $\fq$ by $\s:X^{q+1}+Y^{q+1}+Z^{q+1}=0$, that is, $\s$ is a cone over a Hermitian plane curve. As we saw above, in this case the general fiber of the Gauss map $\Gamma$ of $\s$ is a line, and then $\Gamma$ is not a finite map. However, in this example $\s$ is $\fqs$-Frobenius nonclassical and $(0:0:0:1)$ is a singular point of $\s$.

We claim that the classical result on plane $\fq$-Frobenius nonclassical curves can be generalized to $\fq$-Frobenius nonclassical hypersurfaces satisfying (A). In order to prove this, we need first a technical result.

\begin{lem}\label{lt}
Let $\s \subset \p^n$ be an absolutely irreducible nondegenerate hypersurface defined over a perfect field $K$ of characteristic $p>0$. Assume that $K(\s)=K(x_1,\ldots,x_n)$ where not all $x_i$ are powers of $p$ and let $\Psi=(u_0:u_1:\cdots:u_n):\s \dashrightarrow \p^n$ be a generically finite separable rational map. Assume further that
\begin{equation}\label{eqpi}
\sum_{i=0}^{n}x_i^{p^\ell}u_i=0
\end{equation}
for some $\ell>0$, where $x_0=1$. Then $\overline{\Psi(\s)}$ is nondegenerate and $K(\s)=K(u_1/u_0,\ldots,u_n/u_0)$.  Equivalently, $\Psi$ is birational.
\end{lem}
\begin{proof}
Let $\p^n_Y$ denote the target projective space of $\Psi$, that is, $\Psi:\s \dashrightarrow \p^n_Y$, and let $(Y_0:\cdots:Y_n)$ be the homogeneous coordinates of $\p^n_Y$. Set
$$
\mathcal{W}:=\overline{\Psi(\s)} \subset \p^n_Y.
$$
Since $\Psi$ is generically finite, we have that $\dim \mathcal{W}=\dim \s=n-1$. Thus, $\mathcal{W}$ is a hypersurface in $\p^n_Y$, possibly, a hyperplane. At a general point $P \in \s$, by equation \eqref{eqpi} we have
$$
\sum_{i=0}^{n}x_i(P)^{p^\ell}u_i(P)=0.
$$
This means that the point $\Psi(P)=(u_0(P):\cdots:u_n(P))$ lies on the hyperplane $H_P \subset \p^n_Y$ defined by
$$
H_P: \sum_{i=0}^{n}x_i(P)^{p^\ell}Y_i=0.
$$
Now, since $\Psi$ is generically finite and separable, then its differential $d\Psi_P=(du_0,\ldots,du_n)$ is generically a isomorphism of $\kk$-vector spaces
$$
d\Psi_P: T_P\s \xlongrightarrow{\ \ \cong \ \ } T_{\Psi(P)}\mathcal{W},
$$
see e.g. \cite[Proposition 5.51]{Mi}. In turn, differentiating both sides of equation \eqref{eqpi} provides
$$
\sum_{i=0}^{n}x_i^{p^\ell}du_i=0
$$
as $d(x_i^{p^\ell})=0$. This means that the hyperplane $H_P$ contains the image of the tangent space $T_P\s$ under the differential $d\Psi_P$, i.e., $H_P \supset T_{\Psi(P)}\mathcal{W}$. However, $\mathcal{W}$ is a hypersurface, which gives that $H_P = T_{\Psi(P)}\mathcal{W}$. Hence the projective coordinates of the image of  $\Psi(P)$ by the Gauss map $\Gamma_\mathcal{W}$ of $\mathcal{W}$ are $(x_0(P)^{p^\ell}:\cdots: x_n(P)^{p^\ell})$. In other words, we conclude that
$$
\Gamma_\mathcal{W} \circ \Psi =\Phi_{p^\ell} \big\vert_{\s}.
$$
Since $\Phi_{p^\ell} \big\vert_{\s}$ is purely inseparable, we have that $\Gamma_\mathcal{W} \circ \Psi$ has separable degree $1$. Therefore
$$
1=\deg_s(\Gamma_\mathcal{W} \circ \Psi )=\deg_s(\Gamma_\mathcal{W})\deg_s(\Psi) \Lra \deg_s(\Psi)=1.
$$
Since $\Psi$ is separable, we obtain that $\deg(\Psi)=\deg_s(\Psi)=1$, which means that $\Psi$ is birational.

Finally, if $\mathcal{W}$ were a hyperplane, then its Gauss map $\Gamma_\mathcal{W}$ would be constant. Then, $\Gamma_\mathcal{W} \circ \Psi$ would be constant as well, contradicting the fact that $\Phi_{p^\ell} \big\vert_{\s}$ is not constant. Therefore, $\overline{\Psi(\s)}$ is nondegenerate.

\end{proof}


We now state the main result of this section.

\begin{thm}\label{55}
Let $\s \subset \p^n$ be an absolutely irreducible nondegenerate hypersurface  defined over $\fq$ with Gauss map $\Gamma$. If $\s$ is $\fq$-Frobenius nonclassical and it satisfies (A) with  $p^r<q$, then $\Gamma$ is purely inseparable. Moreover, in this case, we have that $\s=\s^{\prime \prime}$ and if $\Gamma^\prime$ is the Gauss map of $\s^\prime$, then $\Gamma^\prime \circ \Gamma=\Phi_q$.
\end{thm}
\begin{proof}
If $x_0,\ldots,x_n \in \fq(\s)$ are the projective coordinate functions of $\s$,  (A) implies that
$$
\sum_{i=0}^{n}U_i^{p^r}x_i=0,
$$
and $\Gamma=(U_0^{p^r}:\cdots:U_n^{p^r})$ for some $r>0$. More precisely, $$\Gamma=\Phi_{p^r} \circ \Psi,$$ where  $\Psi=(U_0:\cdots:U_n)$ is a separable rational map. If $q=p^h$, we have from the Frobenius nonclassicality of $\s$ that 
$$
\sum_{i=0}^{n}F_i x_i^q=0 \ \ \Lra \ \  \sum_{i=0}^{n}U_i^{p^r}x_i^q=0 \ \ \Lra \ \ \sum_{i=0}^{n}U_ix_i^{p^{h-r}}=0.
$$
Since $\Psi$ is generically finite and separable, from Lemma \ref{lt} we conclude that $\overline{\Psi(\s)}$ is nondegenerate and $\fq(\s)=\fq(\Psi(\s))$. In particular, the rational map $\Delta=(x_0:\cdots:x_n)$ is defined in $\overline{\Psi(\s)}$, and since for a general point $P \in \s$ we have $x_i(U_0(P) :\cdots: U_n(P))=x_i(P)$ for all $i$, we conclude that $(\Delta \circ \Psi)(P)=P$, that is, 
$$
\Delta \circ \Psi=\operatorname{Id_\s}.
$$
Now, from the equality $\sum\limits_{i=0}^{n}U_ix_i^{p^{h-r}}=0$ we have that 
$$
\Gamma_{\overline{\Psi(\s)}}=(x_0^{p^{h-r}}: \cdots: x_n^{p^{h-r}})=\phi_{p^{h-r}}\circ \Delta,
$$
where $\Gamma_{\overline{\Psi(\s)}}$ stands for the Gauss map of $\overline{\Psi(\s)}$. On the other hand, since $\s^\prime=\Phi_{p^r}(\overline{\Psi(\s)})$, we obtain $\s^{\prime \prime}=\Phi_{p^r}(\overline{\Psi(\s)}^\prime)$, where $\overline{\Psi(\s)}^\prime$ is the dual of $\overline{\Psi(\s)}$. Hence we generically have that
$$
\Gamma^\prime=\Phi_{p^r} \circ \Gamma_{\overline{\Psi(\s)}}\circ \Phi_{p^r}^{-1}.
$$
Therefore, 
\begin{eqnarray}
\Gamma^\prime \circ \Gamma & =& \left(\Phi_{p^r} \circ \Gamma_{\overline{\Psi(\s)}}\circ \Phi_{p^r}^{-1} \right)\circ \left(\Phi_{p^r} \circ \Psi\right) \nonumber \\
&=&\Phi_{p^r} \circ\left( \phi_{p^{h-r}}\circ \Delta\right) \circ \Psi \nonumber \\
&=& \Phi_q \circ \left( \Delta \circ \Psi \right)=\Phi_q. \nonumber
\end{eqnarray}

\[
\begin{tikzcd}[
  row sep=1.5cm,    
  column sep=1.5cm  
]
\s \arrow[r, "\Gamma", dashed] \arrow[dr, "\Psi"',dashed] & \s^\prime \arrow[r, "\Gamma^\prime", dashed] & \s'' \\
& \overline{\Psi(\s)} \arrow[u, "\Phi_{p^r}"'] \arrow[r, "\Gamma_{\overline{\Psi(\s)}}"', dashed] \arrow[dr, "\Delta"', dashed] &\overline{\Psi(\s)}^{\prime} \arrow[u, "\Phi_{p^r}"'] \\
&& \overline{\Delta(\Psi(\s))} = \s \arrow[u, "\Phi_{p^{h-r}}"']
\end{tikzcd}
\]
Furthermore, $\deg_s(\Gamma^\prime)\deg_s(\Gamma)=\deg_s(\Phi_q)=1$. Hence, $\Gamma$ is purely inseparable. Finally,  $\s^{\prime\prime}=(\Gamma^\prime \circ \Gamma)(S)=\Phi_q(S)=S$, which finishes the proof.

\end{proof}

\begin{rem}
By \cite[Corollary I.2.8]{Za}, a sufficient condition for the Gauss map of $\s$ to be generically finite is that $S$ is nonsingular.   
\end{rem}

\section{Some open problems}

In this section, we present some interesting questions arising from our results. The first one relates to the inseparability of the Gauss map on a hypersurface $\s$ versus the inseparability of its restriction to curves lying on $\s$. 

Indeed, Theorem \ref{fnc} proves that if $\s \subset \p^n(\fqc)$ is an $\fq$-Frobenius nonclassical irreducible hypersurface such that $p>n$, then every curve $\xx \subset \s$ defined over $\fq$ such that $\xx$ is not contained in the singular locus of $\s$ is $\fq$-Frobenius nonclassical, provided that the Gauss map of $\s$ restricted to $\xx$ is inseparable. However,  Example \ref{resep} shows that  even if the Gauss map $\Gamma$ of a hypersurface $\s$ is inseparable, the restriction of $\Gamma$ to a curve $\xx \subset \s$ can be separable. Also, in this case, $\xx$ may be  classical. It should be noted that the hypersurface in Example \ref{resep} is not Frobenius nonclassical, while the Gauss map of  every Frobenius nonclassical hypersurface  presented in this paper admits $p$-th powers as coordinate functions.  As a matter of fact, the same holds true for  every Frobenius nonclassical hypersurface known to the authors.
In particular, the restriction of these Gauss maps to every curve contained in such Frobenius nonclassical hypersurfaces is clearly inseparable. This naturally raises the following problem.


\begin{prob}
Let $\s \subset \p^n(\fqc)$ be an $\fq$-Frobenius nonclassical hypersurface, where $\Char(\fq)=p>n$. Is it true that the Gauss map of $\s$ always admits as coordinate functions $p$-th powers? If not, is it true that if $\xx \subset \s$ is defined over $\fq$, then the Gauss map of $\s$ restricted to $\xx$ is always inseparable?
\end{prob}

At the heart of the above problem lies the fact that, while the definition of Frobenius nonclassical hypersurfaces is quite easy and carries a deep geometric meaning, it is indeed quite hard to work with. This goes to the extent that, even in the plane case we do not have a complete classification of $\fq$-Frobenius nonclassical curves. 

This makes one wonder about the existence of necessary and sufficient conditions for nonclassicality that are somehow easier to check than the actual definition.  Proposition \ref{class} provides one such criterion for nonclassicality of curves in $\mathbb{P}^{n}(\mathbb{K})$, provided that $p >n$ and $\mathbb{K}$ is algebraically closed. Based on this,  we raise the following question on $\fq$-Frobenius nonclassical curves $\xx$.

\begin{prob}
Can we give conditions to an $\fq$-Frobenius nonclassical curve $\xx$ so that the converse of Theorem \ref{fnc} holds?
\end{prob}
 \bmhead{Acknowledgements}
The authors thank the anonymous referee for the careful reading of the manuscript and for the many valuable comments and suggestions, which greatly improved the presentation of the paper. In particular, the current proof of Lemma \ref{lt} is largely based on their comments. 
The research of Nazar Arakelian was partially supported by grant 2023/03547-2, S\~ao Paulo Research Foundation (FAPESP). The research of Pietro Speziali was partially supported by  FAEPEX grant 2597/25.

\section*{Declarations}

\begin{itemize}
\item Funding Mentioned in Acknowledgements
\item Conflict of interest/Competing interests: Not applicable 
\item Ethics approval and consent to participate: Not applicable
\item Consent for publication: granted
\item Data availability: not applicable
\item Materials availability: not applicable
\item Code availability: not applicable
\item Author contribution: the two authors contributed equally.
\end{itemize}

\bibliographystyle{amsplain}

\end{document}